\documentclass[12pt]{amsart}
\usepackage[dvips]{graphicx}
\usepackage{amsmath,graphics}
\usepackage{amsfonts,amssymb}
\usepackage{xypic}
\theoremstyle{plain}
\usepackage{color}
\newtheorem{theorem*}{Theorem}
\newtheorem*{lemma*} {Lemma}
\newtheorem{corollary*}{Corollary}
\newtheorem{proposition*}[theorem*]{Proposition}
\newtheorem*{conjecture*}{Conjecture}
\newtheorem{theorem}{Theorem}[section]
\newtheorem{lemma}[theorem]{Lemma}
\newtheorem*{theorem1*}{Theorem 1}
\newtheorem*{theorem2*}{Theorem 2}
\newtheorem*{theorem3*}{Theorem 3}
\newtheorem*{theorem4*}{Theorem 4}
\newtheorem*{theorem5*}{Theorem 5}
\newtheorem*{theorem6*}{Theorem 6}
\newtheorem*{corollary1*}{Corollary 1}
\newtheorem*{corollary2*}{Corollary 2}
\newtheorem*{corollary3*}{Corollary 3}
\newtheorem*{corollary4*}{Corollary 4}
\newtheorem*{corollary5*}{Corollary 5}
\newtheorem*{corollary6*}{Corollary 6}
\newtheorem{corollary}[theorem]{Corollary}
\newtheorem{proposition}[theorem]{Proposition}

\newtheorem{question}[theorem]{Question}

\theoremstyle{remark}
\newtheorem*{remark}{Remark}
\newtheorem*{remarks}{Remarks}
\newtheorem*{definition}{Definition}

\newtheorem{example*}{Example}
\newtheorem*{claim}{Claim}

\newtheorem*{observation}{Observation}

\theoremstyle{definition}

 \def\Q{\Bbb{Q}}  \def\Z{\Bbb{Z}} \def\R{\Bbb{R}} 
\def\N{\Bbb{N}}   \def\ll{\langle} \def\rr{\rangle}
 \def\a{\alpha}   \def\bp{\begin{pmatrix}}
\def\sm{\setminus} \def\ep{\end{pmatrix}} \def\bn{\begin{enumerate}} 
  \def\div{\mbox{div}} \def\en{\end{enumerate}}
\def\ba{\begin{array}} \def\ea{\end{array}} 
 \def\S{\Sigma}  \def\a{\alpha} \def\b{\beta} \def\ti{\wti}

\def\be{\begin{equation}} \def\ee{\end{equation}} 
   
 \def\hom{\mbox{Hom}}

\def\zt{\Z[t^{\pm 1}]}    
\def\w{\omega}   
    \def\fr12{\frac{1}{2}} \def\z12{\Z[\fr12]}

\def\k{\kappa}

\def\wti{\widetilde}
\def\spinc{\operatorname{Spin}^c}

\def\cc{\mathfrak{c}}

\begin{document}
\title{Minimal genus in 4--manifolds with a free circle action}
\author{Stefan Friedl}
\address{Mathematisches Institut, Universit\"at zu K\"oln, Germany}
\email{sfriedl@gmail.com}
\author{Stefano Vidussi}
\address{Department of Mathematics, University of California,
Riverside, CA 92521, USA} \email{svidussi@math.ucr.edu}
\thanks{The second author was partially supported by NSF grant \#0906281.}
\date{\today}
\begin{abstract}
 Let $N$ be a closed irreducible $3$--manifold and assume $N$ is not a graph manifold.
We improve for all but finitely many $S^1$--bundles $M$ over $N$ the adjunction inequality for the minimal complexity of embedded surfaces. This allows us to completely determine the minimal complexity of embedded surfaces in all but finitely many $S^1$--bundles over a large class of $3$--manifolds.
\end{abstract}
\maketitle

\section{Introduction and main results}

%=======================================================
\subsection{Complexity of surfaces in 4--manifolds}

Given a compact surface $\Sigma $ with connected components $\Sigma_1,\dots,\Sigma_{l}$ its \textit{complexity} is defined as \[ \chi_{-}(\Sigma) = \sum_{i=1}^{l} \mbox{max}(-\chi(\Sigma_{i}),0). \]
Given a smooth closed 4--manifold  $M$ and
 $\a\in H_2(M)$ we define
 \[ x(\a):= \mbox{min} \{ \chi_{-}(\Sigma)\, |\,  \Sigma \subset M\mbox{ embedded surface which represents }\a\}. \]
It is a classical problem  to determine $x$ for a given 4--manifold.
A key tool for finding lower bounds on $x$ comes from the adjunction inequality (see \cite{MST97}),
which states that when $M$ has $b_{+}(M) > 1$, if $\Sigma\subset M$ is a connected embedded surface  with $g(\Sigma) > 0$ and non--negative intersection number,
then
\begin{equation} \label{usual}  \chi_-(\Sigma) \geq [\S] \cdot [\S] + \kappa \cdot [\Sigma]. \end{equation}
Here $\kappa \in H^2(M)$ is a Seiberg--Witten \textit{basic class} of $M$, i.e. the Chern class of a $\spinc$--structure for which the Seiberg--Witten invariant is nontrivial.

It is natural to ask whether the inequality of Equation (\ref{usual}) gives the best possible bound on complexity, and if basic classes determine the complexity function of a $4$--manifold. In this regard it is instructive to look at the comparable situation for 3--manifolds.
Given a 3--manifold $N$ and a class $\sigma\in H_2(N)$ we
 define the Thurston norm of $\sigma$ (see \cite{Th86}) to be
\[ \|\sigma\|_T := \mbox{min} \{ \chi_{-}(\Sigma)\, |\, \Sigma \subset N\mbox{ embedded surface which represents } \sigma\}. \]
The $3$--dimensional Seiberg--Witten basic classes give again a lower bound on the complexity
of surfaces. More precisely under the assumption that $b_1(N)>1$ and $\sigma\in H_2(N)$ the adjunction inequality for 3--manifolds states that
\begin{equation} \label{equ:auckly} \|\sigma\|_T \geq \kappa \cdot \sigma, \end{equation} where $\kappa \in H^2(N)$ is a Seiberg--Witten \textit{basic class} of $N$ (see \cite{Kr98} or \cite{Au96} for details).

It is well known that in general there exists no basic class which turns  (\ref{equ:auckly}) into an equality.
But an equality can be obtained, at least for a large class of $3$--manifolds, using monopole classes. More precisely Kronheimer and Mrowka \cite{KM97} showed that if $N$ is irreducible
given any $\sigma\in H_2(N)$ there exists a Seiberg--Witten  \emph{monopole class} $\kappa$ of $N$ such that
\[  \|\sigma\|_T = \kappa \cdot \sigma.\]
By the respective definitions, a Seiberg--Witten basic class is also a monopole class, but the converse does in general not hold, and it is easy to give examples where the inequality of (\ref{equ:auckly}) is strict.

We now return to the study of 4--manifolds. The above discussion suggests that the adjunction inequality for 4--manifolds  in general is not sharp.
This is indeed the case, but it is not easy to pin down
examples where this occurs.
 The most relevant instance in this sense comes from Kronheimer's  refined adjunction inequality for manifolds of the form $S^1 \times N^3$. More precisely, Kronheimer \cite[Corollary~7.6]{Kr98} \cite{Kr99} proved that if $N$ is an irreducible
 closed 3--manifold such that the Thurston norm does not identically vanish, then
given any $\a\in H_2(S^1\times N)$
the following inequality holds:
\be \label{equ:kronheimer}  x(\a)\geq |\alpha\cdot \a|+\|p_*\a\|_T,\ee
where $p\colon S^1\times N \to N$ denotes the projection map.
It is easy to find examples where this refined adjunction inequality becomes an equality
while the adjunction inequality in (\ref{usual}) is strict. A second type of results, more closely connected with the approach of this paper, is contained in \cite{FV09}, where the authors used
 Seiberg Witten invariants of finite covers to  improve the adjunction inequality for manifolds $M$ that admit a free circle action.

%=======================================================
\subsection{Complexity of surfaces in 4--manifolds with a free $S^1$--action}

In this paper we will show that in most cases the Seiberg--Witten invariants
of finite covers contain enough information to recover Kronheimer's inequality for closed irreducible $3$--manifolds (with the exclusion of non--npc graph manifolds, i.e. graph manifolds that do not admit a nonpositively curved metric), without any assumption on the Thurston norm of $N$. Moreover, we will extend that inequality for circle bundles  $p\colon M \to N$  with Euler class $e \in H^{2}(N) \setminus \Xi_N$, where  $\Xi_N$ is a \emph{finite} subset (described in Section \ref{section:defxn}) determined by the Thurston norm of $N$, that does not contain any torsion elements.
Precisely, the main result of this paper is then the following:

\begin{theorem} \label{main} \label{mainthm} Let $N$ be a closed irreducible $3$--manifold and assume $N$ is not a non--npc graph manifold. Let $p\colon M \to N$ be the  circle bundle with Euler class $e \in H^{2}(N) \setminus \Xi_N$. Then for any $\a\in H_2(M)$ we have
  \begin{equation} \label{refadj}  x(\a) \geq |\a \cdot \a| +  \| p_{*}\a\|_T. \end{equation}

  \end{theorem}

\begin{remark}
1. This theorem was proved by Kronheimer \cite{Kr99} for the case $e=0$, i.e. $M=S^1\times N$, under the assumptions mentioned  in the previous section.\\
2.  If the intersection number $\a\cdot \a$ is zero, then the conclusion of the theorem,  without any restrictions on $N$, also follows immediately from the following result of  Gabai \cite[Corollary~6.18]{Ga83}:
\[ \|\sigma\|_T= \mbox{min} \{ \chi_{-}(\Sigma)\, |\, \Sigma \subset N\mbox{ \emph{immersed} surface which represents } \sigma\}. \]
\end{remark}

In fact, in this paper we will show in Lemma \ref{above} that the conclusion of Theorem \ref{main} hold for all closed irreducible $3$--manifold whose fundamental group is virtually RFRS, i.e. admits a finite index subgroup which is RFRS. This assumption, introduced by Agol in \cite{Ag08} (and discussed in Section \ref{section:agol}) is the key topological condition which allows us to get a handle on Seiberg--Witten invariants  of finite covers of $N$ via Agol's virtual fibering theorem. Theorem \ref{main} is then consequence of this result and the following

\begin{proposition} \label{prop:liuconj}
Let $N$ be a closed irreducible 3--manifold, then $\pi_1(N)$ is virtually RFRS, unless $N$ is a non--npc graph manifold.
\end{proposition}

Proposition \ref{prop:liuconj} combines work of Wise  (\cite{Wi09}, \cite[Section~14]{Wi12} and \cite[Section~15.2]{Wi11}) and Agol (\cite{Ag12}) for the case of hyperbolic manifolds, Liu  (\cite{Li11}) for the case of graph manifolds, and finally Przytycki--Wise (\cite{PW12}) for the case of mixed manifolds, where the authors prove that the fundamental group of the manifolds under discussion are
\textit{virtually special}.  In turn, this condition (whose precise definition is of no concern to us) implies by \cite{HW08} and  \cite[Corollary~2.3]{Ag08} that the fundamental group is virtually RFRS. It is worth mentioning that partial results in this direction had been obtained in Bergeron and Wise \cite{BW09} and Chesebro, Deblois and Wilton \cite{CDW09}.

%=================================================
\subsection{Complexity in $S^1$--bundles over nice 3--manifolds}\label{section:nice}

We now investigate some conditions under which the refined adjunction inequality of Theorem \ref{main} is sharp, i.e. the complexity function is determined by the Thurston norm of $N$. In the same vein, we examine whether the complexity function exhibits properties that correspond to those of the Thurston norm.

\begin{definition}
Let $N$ be a closed irreducible 3--manifold and assume $N$ is not a non--npc graph manifold. Denote $H:=H_1(N;\Z)/\mbox{torsion}$.
 We say $N$ is \emph{nice} if its multivariable Alexander polynomial $\Delta_N=\sum_{h\in H}a_h\cdot h\in \Z[H]$
 (see e.g. \cite{FV10} for the definition) has the property that
 $\phi(\Delta_N):=\sum_{h\in H}a_h\cdot t^{\phi(h)}\in \zt$ is non--zero for any epimorphism $\phi:H\to \Z=\ll t\rr$.
\end{definition}

It is easy to see that a `generic' non--zero polynomial $\Delta\in \Z[H]$
has the property that all possible one--variable specializations $\phi(\Delta_{N})$ are non--zero.
We refer to Section \ref{section:sharp} for details.

The following result, which is inspired by \cite[Corollary~2]{Kr99} in the product case, follows immediately from Theorem \ref{main} and Lemmas
 \ref{lem:starmin} and \ref{lem:getstar}.
This corollary completely determines the complexity function of all but finitely $S^1$--bundles over nice 3--manifolds.

\begin{corollary} \label{cor:mingen}
Let $p\colon M \to N$ be an $S^1$--bundle over a nice 3--manifold $N$.
If $e\not\in \Xi_N$, then  for any  $\a \in H_2(M)$ the following equality holds:
\[  x(\a) = |\a \cdot \a| +  \| p_{*}\a\|_T.\]
 \end{corollary}

Before we continue in our discussion we want to recall that the Thurston norm has the following properties
(see \cite{Th86} and \cite{Ga83}):
\bn
\item linearity on rays,
\item triangle inequality,
\item multiplicativity under finite covers, i.e.  given a finite cover $\pi \colon {\wti N} \to N$ and $\sigma\in H_2(N)$ we have
 \[ \| \pi^{*} \sigma \|_T = {\mbox{deg} \, \pi} \| \sigma \|_T,\]
 where $\pi^{*} \colon H_2(N) \to H_{2}({\wti N})$ is the transfer map in homology.
\en
Note that  (1) and (2) imply that the Thurston norm on $H_2(N;\Z)$ extends to a   (semi) norm on  $H_2(N;\R)$.
We now consider
the Thurston norm unit ball
\[ B_{N}:=\{\sigma \in H_2(N;\R)\,|\,\,  \|\sigma\|_T\leq 1\} \subset H_2(N;\R).\]
By \cite{Th86} this  is a finite, convex, rational polyhedron.
We say $\sigma\in H_2(N;\R)$ is a \emph{fibered class} if there exists a non--degenerate closed 1--form $\psi$ on $N$
such that $\sigma=PD([\psi])$.
Thurston \cite{Th86} proved the following relationship between the Thurston norm ball and fibered classes:
\bn
\item[(4)] A fibered class $\sigma$ lies in the cone on an open top dimensional face $F$ of $B_N$ and  any other class in the open cone on $F$ is also fibered.
\en
It is a natural question whether the complexity function $x$ on 4--manifolds exhibits similar naturality properties.
In fact, Corollary \ref{cor:mingen} asserts that for most $S^1$--bundles over a nice 3--manifold
the function
\[ x'(\a):=x(\a)-|\a\cdot \a|\]
satisfies Properties (1) and (2), i.e. $x'$ is linear on rays and  it satisfies the triangle inequality.

Regarding Property (3), we can prove the following theorem, that provides a family of manifolds that satisfy the properties requested in Question 7.8
 of \cite{Kr98}. In the statement we are constrained to exclude, this time, an infinite set of choices $\Theta_{N} \subset H^{2}(N)$ for the Euler class. This set is given by classes which are multiples, in rational cohomology, of elements in $\Xi_{N}$.

\begin{theorem} \label{kronansintro} Let $N$ be a nice $3$-manifold.
Let  $p \colon  M \to N$ be a circle bundle of Euler class $e \in H^2(N) \setminus \Theta_{N}$. Then the minimal complexity function $ x$ is multiplicative under finite covers.  \end{theorem}
We thus see that in many cases $x'$ is indeed well behaved.

Regarding Property (4), it is well--known that symplectic classes of 4--manifolds in many ways play the r\^ole of
fibered classes on 3--manifolds. Here we say that $\a\in H_2(M;\R)$ is \emph{symplectic} if there exists a symplectic form $\w$ on $M$ such that $\a=PD([\w])$.
In \cite{FV12a,FV12c} we showed that a class $\a\in H_2(M;\R)$  is symplectic if and only if it is contained in the positive cone of $ H_2(M;\R)$ and $p_{*} \a\in H_2(N;\R)$ is contained in the fibered cone of $N$ (in particular, $N$ fibers over $S^1$). Combining this with Corollary \ref{cor:mingen} we get the following.

\begin{theorem}\label{thm:symplecticcone}

Let $p\colon M \to N$ be an $S^1$--bundle over a nice 3--manifold.
If $e\not\in \Xi_N$, then $x'\colon H_2(M;\Z)\to \Z_{\geq 0}$ gives rise to a seminorm $x'$ on $H_2(M;\R)\to \R_{\geq 0}$, which has the following property:
If $\a\in H_2(M)$ is a symplectic class, then it lies on the cone over an open top dimensional face $F$ of the unit norm ball of $x'$,
and any other  class in the cone on $F$ is also  symplectic.
\end{theorem}

\begin{remarks}
1. It is a very interesting question whether the results of this section hold for \textit{any}  $S^1$--bundle over nice 3--manifold. It is likely that the exclusion of certain Euler classes is only due to the limitations of our method. \\
2. The restriction to nice 3--manifolds is not optimal, but  the discussion of certain examples in Section \ref{section:sharp}  suggests that there may exist 3--manifolds $N$ such that equality in Corollary \ref{cor:mingen}
 does not even hold for $S^1\times N$, and similarly, that all the other theorems also fail to hold for such examples. In particular the examples suggest that $x'$ may not always well behaved.
\end{remarks}

In  \cite{FV12b} we use related ideas to show that twisted Alexander polynomials detect the Thurston norm of most irreducible 3--manifolds.

%==========================================================
\subsection*{Acknowledgment.} We wish to thank Genevieve Walsh for a very helpful conversation.

%==========================================================
\section{Preliminaries}\label{section:prelim}
\label{section2} We start by recalling some elementary facts about the algebraic topology of a
$4$--manifold $M$ that is the total space of a circle bundle $p \colon M \to N$ over a $3$--manifold with nontorsion Euler class $e \in H^2(N)$.
The Gysin sequence (with either integer or real coefficients) reads
\[ \xymatrix{  H^{0}(N) \ar[d]^\cong \ar[r]^{\cup \hspace*{1pt} e}& H^2(N) \ar[d]^\cong \ar[r]^{p^{*}}&
H^2(M) \ar[d]^\cong \ar[r]^{p_{*}} &H^{1}(N) \ar[d]^\cong \ar[r]^{\cup \hspace*{1pt} e} &H^3(N)\ar[d]^\cong \\
H_3(N)\ar[r]^{\cap e}&H_1(N)\ar[r]&H_2(M)\ar[r]^{p_*}&H_2(N)\ar[r]^{\cap e}&H_0(N),}
\]
where $p_* \colon H^{2}(M) \to H^{1}(N)$ denotes integration along the fiber.
Note that, in particular, we have \[ 0 \to
\langle e \rangle \to H^2(N) \xrightarrow{p^{*}} H^2(M) \xrightarrow{p_{*}} \mbox{ker}( \cdot e)
\to 0, \] where we have denoted by $\langle e \rangle$ the cyclic subgroup of $H^2(N)$
generated by the Euler class and by $\mbox{ker}( \cdot e)$ the subgroup of $H^1(N)$ whose pairing
with the Euler class vanishes. It follows that $b_2(M)=2b_1(N) - 2$.
 It is not difficult to verify that $\operatorname{sign}(M) = 0$, hence $b_2^+(M)=b_1(N) - 1$.

We will also consider finite covers $\pi_{M}\colon {\wti M} \to M$ of $M$, and we collect here some results that will be of use. The manifold ${\wti M}$ carries a free circle action whose orbit space we denote by ${\wti N}$, so we have an $S^1$--bundle ${\wti p}\colon {\wti M} \to {\wti N}$. The latter $3$--manifold is itself a finite cover $\pi_{N}\colon {\wti N} \to N$, so that the $S^1$--bundle map and the covering map are related by the commutative diagram
\[ \label{diag:cover} \xymatrix{ {\wti M} \ar[d]^{\wti{p}}\ar[r]^{\pi_{M}} & M \ar[d]_p \\ {\wti N} \ar[r]^{\pi_{N}} & N,}. \]

Note that, in general, ${\wti p}\colon {\wti M} \to {\wti N}$ is \emph{not} the pull--back bundle under $\pi_{N}$ of $p\colon M \to N$. In fact, a fiber of ${\wti p}$ is the connected cover of degree $q$ of the fiber of $p$, where $q$ is the index of $\pi_{1}(S^1) \cap \pi_{1}(\wti M) \leq_{fi} \pi_{1}(S^1)$, here $\pi_{1}(S^1) \leq \pi_1(M)$ denotes the (central) subgroup carried by the fiber of $p\colon M \to N$. (The case of pull--back bundle corresponds to $q = 1$.) The degrees of $\pi_{M}$ and $\pi_{N}$ are then related by the formula $\mbox{deg} \, \pi_{M} = q \, \mbox{deg} \, \pi_{N}$.
The Euler class $\wti e \in  H^2({\wti N})$ of the $S^1$--bundle ${\wti p}\colon {\wti M} \to {\wti N}$ and the
Euler class $e \in H^{2}(N)$ of the $S^1$--bundle $p\colon M \to N$ are related by the equation (see e.g. \cite{Br94}) \be \label{equ:coverq} q  \wti{ e} = \pi^{*}_{N} e \in H^{2}({\wti N}). \ee
 As $\pi_{N}^{*} \colon  H^{2}(N;\Q) \to H^{2}({\wti N};\Q)$ is injective, ${\wti e}$ is nontorsion.

We will make use of the following inequalities. Given $\a \in H_{2}(M)$ and $\sigma \in H_{2}(N)$, we have
\be \label{equ:facts} \ba{rcl} x({\pi_{M}^{*} \a})&\leq &  \mbox{deg} \, \pi_{M} \cdot x(\a),\\
|{\pi_{M}^{*} \a} \cdot {\pi_{M}^{*} \a}|&=& \mbox{deg} \, \pi_{M} \cdot | \a \cdot \a |,\\
\| \pi_{N}^{*} \sigma\|_T&=& \mbox{deg} \, \pi_{N}\cdot  \|\sigma\|_T.\ea \ee
The first two statements are trivial and the third is  the multiplicativity of the Thurston norm under finite covers.

 Finally, we are interested in the Seiberg--Witten basic classes of $M$. Conveniently, we will be able to restrict the study of Seiberg--Witten invariants to the case where $b_{2}^{+}(M) > 1$, avoiding this way the technicalities that appear for $b_{2}^{+}(M) = 1$. Under that assumption, Baldridge has proven in \cite{Ba01} that the Seiberg--Witten invariants can be written in terms of $3$--dimensional Seiberg--Witten invariants as follows:

\begin{theorem} \label{bald} \label{thm:baldridge} \textbf{\emph{(Baldridge)}} Let $N$  be a closed $3$--manifold with $b_1(N) > 2$, and let $p\colon M \to N$ be the circle bundle with nontorsion Euler class $e \in H^2(N)$. Then the Seiberg-Witten invariant of the pull--back $p^{*} \cc \in \spinc(M)$ of a $\spinc$--structure $\cc  \in \spinc(N)$  is given by \[ SW_{M}(p^{*}\cc)
= \sum_{l\in \Z} SW_{N}(\cc+le) \in \Z \] and vanishes for all other  $\spinc$--structures.
\end{theorem}

\begin{remarks}
1.  It is well--known that in the case $e=0$, i.e. $M=S^1\times N$ the equality  $ SW_{M}(p^*\cc)
=  SW_{N}(\cc)$ holds.\\
2.  Note that, while this theorem asserts that the set $\mbox{supp}(SW_{M}) \subset \spinc(M)$ is contained in the image, under $p^{*} \colon \spinc(N) \to \spinc(M)$, of $\mbox{supp}(SW_N) \subset \spinc(N)$, it does not imply that the image of a $3$--dimensional basic class must be a $4$--dimensional basic class, as the averaging process in the formula above can (and in examples does) cause the $4$--dimensional invariant to vanish.
\end{remarks}

%=======================================================
\section{The refined adjunction inequality } \label{sec:ref}

%=======================================================
\subsection{The dual Thurston norm ball}
\label{section:defxn}

Let $N$ be a closed 3--manifold. Recall that we denote by $B_N\subset H_2(N;\R)$ the Thurston norm ball.
We now denote by
\[ B_{N}^*:=\{ \xi\in H^2(N;\R)\,|\, \xi\cdot \sigma\leq 1 \mbox{ for all }\sigma\in B_{N}\}\subset H^2(N;\R)=H_1(N;\R),\]
the ball dual to $B_{N}$. We will refer to $B_{N}^*$ as the \emph{dual Thurston norm ball}.
By \cite{Th86} this is a finite, convex, compact polyhedron in $H^2(N;\R)$ such that all vertices lie in the image of $\iota\colon H^2(N)\to H^2(N;\R)$.
Note that the $k$--dimensional faces of $B_{N}$ are in one--to--one correspondence with $(b_1(N)-1-k)$--dimensional faces of $B_{N}^*$.
More precisely, if $\sigma \in H_2(N;\R)$ lies on an open $k$--dimensional face, then
\[ \{ \xi \in H^2(N;\R) \,|\, \xi\cdot \sigma =1\} \]
is a $(b_1(N)-1-k)$--dimensional face of $B_N^*$.

Note that we can recover the Thurston norm of $N$ from $B_N^*$. In fact, if  we denote by $V$ the set of vertices of
$B_{N}^{*}$, then for any $\sigma\in H_2(N;\R)$ the following equality holds:
\[ \|\sigma\|_T = \mbox{max} \{ v \cdot \sigma\, |\, v \in V \}.\]
It is clear that if $H\subset H^2(N;\R)$ is any subset that has the property that
\[ \|\sigma\|_T = \mbox{max} \{ h \cdot \sigma\, |\, h \in H \},\]
for any $\sigma\in H_2(N;\R)$,
then $V\subset H$.

We define now two subsets of $H^2(N)$ that enter in the statement of our main results. Consider first the following subset of $H^2(N;\R)$: \[ \ba{rcl} {\mathcal E}_{N} = \{ w\in H^2(N;\R)\sm \{0\}&|& \mbox{ there exists  $v\in V$ such that }\\
&&\mbox{$v+2w$  lies on an edge of } B_{N}^*\}.\ea \]
Given the canonical map $H^2(N)\to H^2(N;\R)$, we define $\Xi_N \subset H^{2}(N)$ to be the inverse image of ${\mathcal E}_{N}$, and $\Theta_{N} \subset H^{2}(N)$ to be the inverse image of $\R_{+} \cdot {\mathcal E}_{N}$. Neither $\Xi_N$ or $\Theta_{N}$ contain any torsion elements. Note that $\Xi_{N}$ is finite, while $\Theta_{N}$ is not. However, it should be clear that ``most" classes belong to $H^2(N) \setminus \Theta_{N}$.

%=======================================================
\subsection{The adjunction inequality for $S^1$--bundles}

We will later use the following consequence of the adjunction inequality (\ref{usual}):

\begin{theorem}\label{thm:adjs1}
Let $N$  be a closed irreducible $3$--manifold with $b_1(N) > 2$ and let $p\colon  M \to N$ be a circle bundle. Let $\a\in H_2(M)$.
Then the following inequality holds for any basic class $\k$ of $M$:
\[    x(\a)\geq |\a\cdot \a|+ \k \cdot \alpha.\]
\end{theorem}

\begin{proof}
Let $\a\in H_2(M)$ and let $\kappa$ be a basic class of $M$.
Let $\S\subset M$ be an embedded surface which minimizes the complexity of $\a$.
We denote the components of $\S$ by $\S_1,\dots,\S_l$. As $N$ is irreducible, $\pi_{2}(M)$ vanishes, hence all the $\Sigma_{i}$'s must have positive genus.
If $\S_i$ has positive self--intersection, then it follows from
 the adjunction inequality (\ref{usual}) that
\[ \chi_-(\S_i) \geq [\S_i]\cdot [\S_i]+\kappa \cdot [\S_i].\]
On the other hand, if $\S_i$ has negative self--intersection,
then we consider $\S_i$ as a surface in $-{M}$.
It follows from Theorem \ref{thm:baldridge} that $\kappa$ is a basic class for $M$ if and only if $\kappa$ is a basic class for $-{M}$. It is well--known that if $\kappa$ is a Seiberg--Witten basic class, then $-\kappa$ is a Seiberg--Witten basic class. Using these observations one can now easily show that
\[ \chi_-(\S_i) \geq -[\S_i]\cdot [\S_i]+\kappa \cdot [\S_i].\]
We refer to \cite[Corollary~3.2]{FV09} for details.
Since $[\S_i]\cdot [\S_j]=0$ for $i\ne j$ it now follows that
\[ \ba{rcl} x(\a)& = &\sum_{i=1}^l \chi_-(\S_i)\\[2mm]
&\geq &\sum_{i=1}^l \big(|[\S_i]\cdot [\S_i]|+\kappa \cdot [\S_i]\big) \\[2mm]
&\geq &\left|\sum_{i=1}^l [\S_i]\cdot [\S_i]\right|+ \sum_{i=1}^l \kappa \cdot [\S_i]\\[2mm]
&=&\left|\left(\sum_{i=1}^l [\S_i]\right)\cdot \left(\sum_{i=1}^l [\S_i]\right)\right| \,+\,\kappa \cdot \sum_{i=1}^l [\S_i]\\[2mm]
&=& |\a\cdot \a|+ \kappa \cdot \alpha.\ea \]
\end{proof}

In what follows we will get further information on $x$ by applying Theorem \ref{thm:adjs1} to the finite covers of $M$.

%=======================================================
\subsection{Agol's theorem}\label{section:agol}

The topological input for our main result is Lemma \ref{lem:agol}, which is a rather straightforward consequence of Agol's virtual fiberability criterion from \cite{Ag08}. We recall from \cite{Ag08} the following definition: A group $G$ is called \textit{residually finite rationally solvable} or \textit{RFRS} if there
exists a  nested cofinal sequence of normal finite index subgroups $G_i \vartriangleleft G$ such that for any $i$ the map $G_i\to G_i/G_{i+1}$ factors through $G_i\to H_1(G_i)/\mbox{torsion}$. A group is \textit{virtually} RFRS if it has a finite index subgroup that is RFRS.
Agol's theorem can now be stated as follows:

\begin{theorem} \textbf{\emph{(Agol)}} \label{thm:agol}
 Let $N$ be an irreducible  $3$--manifold with virtually RFRS fundamental group. Then every class $\sigma \in H_2(N)$ is virtually quasifibered, i.e. given any $\sigma \in H_2(N)$ there exists a finite cover $\pi \colon {\wti N} \to N$ so that the transfer $\pi^{*} \sigma \in H_2({\wti N})$ is contained in the closure of a fibered cone of the Thurston norm ball of $H_2({\wti N};\R)$.
\end{theorem}

Using Agol's theorem we can now prove the following lemma:

\begin{lemma}\label{lem:agol}
Let $N$ be an irreducible $3$--manifold with virtually RFRS fundamental group.
 Then there exists a finite cover $\pi \colon  {\wti N} \to N$  such that for every nontrivial class $\sigma \in H_2(N;\R)$, the class $\pi^* \sigma \in H_2({\wti N};\R)$ lies in the closure of the cone over a fibered face of the Thurston unit ball $B_{{\wti N}}$. \end{lemma}

\begin{proof}
Let   $\sigma \in H_2(N;\R)$ be a class contained in the cone over a top--dimensional open face of $B_{N}$.
By Agol's theorem there exists a finite cover  $\pi \colon  {\wti N} \to N$  such that $\pi^{*} \sigma$ is quasifibered.
The transfer map $\pi^{*} \colon  H_2(N;\R) \to H_2({\wti N};\R)$ is, up to scale, a monomorphism of normed vector spaces when we endow these spaces with their respective Thurston norm. It follows that the transfer under $\pi$ of any class in the \textit{closure} of the open cone (in $H_2(N;\R)$) determined by $\sigma$ will be quasifibered in ${\wti N}$. For the same reason if a class lies in the closure of a fibered cone, its transfer under further finite covers will enjoy the same property (transfers of fibrations are fibrations). Recall now that the Thurston unit ball of a $3$--manifold is a finite, convex polyhedron, in particular it has finitely many top--dimensional open faces. By picking one class in the cone above each of these faces, and repeatedly applying Agol's theorem to the (transfer of) each such class, we obtain after finitely many steps the cover $\pi \colon  {\wti N} \to N$ of the statement.\end{proof}

This lemma, together with some standard results of Seiberg--Witten theory, allows us to describe the dual Thurston norm (hence, ultimately, the Thurston norm) of $N$ in terms of Seiberg--Witten basic classes of the cover ${\wti N}$. In order to avoid technical complications, we will limit ourselves to the case where $b_1(N) > 1$, that will be sufficient for our purpose.

In the statement of the corollary and in the rest of the paper we often blur
the difference between $H^2(N)$ and its image under the canonical map $H^2(N)\to H^2(N;\R)$.
Even though this map is not necessarily injective this abuse of notation should not pose any problems to the reader.

\begin{corollary} \label{cor:convhull} Let $N$ be an irreducible $3$--manifold with virtually RFRS fundamental group
 and $b_1(N) > 1$, and denote by $\pi \colon  {\wti N} \to N$ the cover of $N$ determined in Lemma \ref{lem:agol}. Then the dual Thurston norm ball $B_{N}^{*} \subset H^2(N;\R)$ is the convex hull of
 \[ H = \{({\mbox{deg} \, \pi})^{-1} \pi_{*} \xi\,| \,\xi = c_1(\cc), \cc \in \mbox{supp}(SW_{\wti N}) \}\subset H^2(N;\R).\] \end{corollary}

\begin{proof} Lemma \ref{thm:agol} guarantees that, for any class $\sigma \in H_2(N;\R)$, the class $\pi^{*} \sigma \in H^{2}(N;\R)$ is contained in the closure of the cone over a fibered face $F$ of $B_{{\wti N}}$. This implies, in particular, that $\pi^{*} \sigma$ is dual to the negative  $-e(F) \in H^2({\wti N})$ of the Euler class of the fibration, so that $-e(F)(\pi^{*} \sigma) =  \|\pi^{*} \sigma \|_T$.
It is well--known that $\pm e(F) \in H^2({\wti N})$ are basic classes of ${\wti N}$ with Seiberg--Witten invariant equal to $1$. (For example this follows from \cite{Ta94} together with the fact that by \cite{Th76} the manifold $S^{1} \times {\wti N}$ admits a symplectic structure with canonical bundle $-e(F) \in H^2({\wti N}) \subset H^2(S^1 \times {\wti N})$.)
As $\xi\cdot \pi^{*} \sigma = \pi_{*} \xi \cdot \sigma$ it follows that for any $\sigma$ we have
\[ ({\mbox{deg} \, \pi}) \|\sigma \|_T =  \|\pi^{*} \sigma \|_T=-e(T)(\pi^* \sigma) \leq \mbox{max} \{  \pi_{*} \xi \cdot \sigma \,|\, \xi = c_1(\cc), \cc \in \mbox{supp}(SW_{\wti N})  \}, \] where the first equality comes from the multiplicativity of the Thurston norm under cover, see (\ref{equ:facts}).
The equality now either follows from the adjunction inequality for 3--manifolds, i.e. inequality (\ref{equ:auckly}), or alternatively from \cite{Ta95}.
\end{proof}

%=======================================================
\subsection{Proof of the main theorem}

As observed in the introduction, Theorem \ref{main} follows immediately combining Proposition \ref{prop:liuconj} with the following

\begin{lemma} \label{above} Let $N$ be a closed irreducible $3$--manifold with virtually RFRS fundamental group. Let $p\colon M \to N$ be the  circle bundle with Euler class $e \in H^{2}(N) \setminus \Xi_N$. Then for any $\a\in H_2(M)$ we have
  \begin{equation}  x(\a) \geq |\a \cdot \a| +  \| p_{*}\a\|_T. \end{equation} \end{lemma}

\begin{proof}
We first consider the case that $e$ is  non--torsion.
Let us start by assuming that $b_1(N) > 2$.
Denote by  $\pi \colon  {\wti N} \to N$ the cover of $N$ determined in Lemma \ref{lem:agol}, whose degree we will denote as $k = \mbox{deg} \, \pi$ for the rest of this proof. By pull--back, we have an induced covering of $M$ with total space  ${\wti M} = \pi^{*} M$. As no risk of confusion arises here, we denote the covering map $\pi \colon {\wti M} \to M$ as well. (Note that both covers have same degree.)

We will write
\[ \ba{rcl} \wti{H}&=& \{\xi = c_1(\cc), \cc \in  \mbox{supp}(SW_{\wti N})\} \subset H^2(\wti{N};\R),\\
H&=& \frac{1}{k}\pi_*(\wti{H})\subset H^2(N;\R),\\
V&=& \mbox{set of vertices of $B_{N}^*$}\subset H^2(N;\R). \ea \]
It follows from the discussion in Section \ref{section:defxn} and Corollary \ref{cor:convhull}  that $V\subset H$ and that the convex hull of $H$ agrees with the convex hull of $V$, i.e. $B_{N}^*$.

Now let $\a\in H_2(M)$. We will write $\sigma=p_*\a$, $\wti{\a}=\pi^*\a$ and $\wti{\sigma}=\pi^*\sigma$.
We define
\[ \ba{rcl} \wti{H}_\sigma&:=&\{ h\in \wti{H}\,|\, h\cdot \wti{\sigma}=\|\wti{\sigma}\|_T\}, \\
{H}_\sigma&:=&\frac{1}{k}\pi_*(\wti{H}_{\sigma})\,\,\,\,\,\mbox{ and }\\
V_\sigma&:=&\{v\in V\,|\, v\cdot \sigma=\|\sigma\|_T\}.\ea \]
 Recall that $\|\wti{\sigma}\|_T=k\|\sigma\|_T$. It follows from the discussion in the proof of Corollary \ref{cor:convhull}  that $\wti{H}_\sigma$ is a non--empty set and that  $V_\sigma\subset H_\sigma$.

Since $b_1(N)>2$ we obtain from Theorem \ref{thm:baldridge} that for any $\cc \in \spinc(\wti N)$ the $4$--dimensional Seiberg-Witten invariant of the class ${\wti p}^{*} \cc \in \spinc(\wti M)$ is given by
\[ SW_{\wti M}({\wti p}^{*} \cc) = \sum_{l\in \Z} SW_{\wti N}(\cc+l\wti{e}),\] where ${\wti e} \in H^2({\wti N})$ is
the Euler class of the
$S^1$--bundle ${\wti p} \colon \wti{M}\to \wti{N}$.

\begin{claim}
If $e\not\in \Xi_N$, then there exists an $h = c_1(\cc) \in \wti{H}_\sigma$ such that $SW_{\wti{M}}(\wti p^*\cc)\ne 0$.
\end{claim}

To prove the claim we start by observing that, as $\wti{p}_*\wti{\a} = {\wti \sigma}$, it follows from the long exact sequence in Section \ref{section:prelim} that
$\wti{e}\cdot \wti{\sigma}=\wti{e}\cdot \wti{p}_*\wti{\a}=0$.
Now, denote by $F$ the face of $B_{N}^*$ which is dual to the open face of $B_{N}$ whose cone contains $\sigma$. Note that $F$ is the convex hull of the vertices in $V_\sigma$;
by the adjunction inequality for 3--manifolds, i.e. inequality (\ref{equ:auckly}),  the face $F$ contains all elements in $H_\sigma$.

Assume by contradiction that for all $\cc \in  \mbox{supp}(SW_{\wti N})$ such that $c_1(\cc) \in \wti{H}_\sigma$ we have $SW_{\wti{M}}(\wti p^*\cc)=0$.
This entails that for any such $\cc \in  \mbox{supp}(SW_{\wti N})$ there exists (at least) one $l\in \Z\sm \{0\}$ such that $\cc+l\wti{e}\in  \mbox{supp}(SW_{\wti N})$.
Since $\wti{e}\cdot \wti{\sigma}=0$, $c_1(\cc+l\wti{e}) = c_1(\cc) + 2l\wti{e}$ lies in fact in $\wti{H}_\sigma$.
Therefore for all $h \in {\wti H}_\sigma$ there exists an $l\in \Z\sm \{0\}$ such that $\frac{1}{k}\pi_{*} h+2l\cdot \frac{1}{k}\pi_*\wti{e}\in H_\sigma$.
As $V_\sigma\subset H_\sigma$, any $v\in V_\sigma$ is of the form $v=\frac{1}{k}\pi_*h$ for some $h\in \wti{H}_\sigma$.
Since furthermore $H_\sigma\subset F$ we see that for all $v\in V_\sigma$ there exists an $l\in \Z\sm \{0\}$ such that $v+2l\cdot \frac{1}{k}\pi_*(\wti{e})\in F$. Therefore the following observation, together with the definition of ${\mathcal E}_N$, implies that $\frac{1}{k}\pi_*(\wti{e})$ lies in ${\mathcal E}_N$.

\begin{observation}
Let $F$ be a  convex compact polyhedron and $w$ a non--zero vector with the property that for any vertex $v$ of $F$ there exists an $l\in \Z\sm \{0\}$ such that $v+lw\in F$. Then one of the edges of $F$ is parallel to $w$ and there exists a vertex $v$ such that $v+w$
 lies on an edge of $F$.
\end{observation}

(For the reader's convenience we give a quick outline of the proof: to each vertex $v$ of $F$ we can assign a well--defined positive (negative) sign if $v+lw\in F$ for some positive (negative) $l$. It is straightforward to see that not all vertices can have the same sign. 
Let $v_+$ and $v_-$ be adjacent vertices such that $v_+$ has positive sign and $v_-$ has negative sign.
Then it is not hard to see
that both $v_{+} +w$ and  $v_{-} -w$ lie on the edge connecting $v_+$ and $v_-$.)

As $q = 1$ (in the notation of Section \ref{section:prelim}), the Euler class of ${\wti M}$ is given by ${\wti e} = \pi^{*} e \in H^2({\wti N})$, hence $\pi_{*} {\wti e} = \pi_{*} \pi^{*} e = k e \in H^{2}(N)$. We conclude therefore that the image $e$ in $H^{2}(N;\R)$ lies in ${\mathcal E}_N$, i.e. $e \in \Xi_{N}$. This concludes the proof of the claim.

Now let $h = c_1(\cc) \in \wti{H}_\sigma$ with  $SW_{\wti{M}}(\wti p^*(\cc))\ne 0$.
By applying Theorem \ref{thm:adjs1} to ${\wti M}$ we  have
\[ \ba{rcl}  x(\wti{\a})&\geq & |\wti{\a}\cdot \wti{\a}|+\wti p^*h\cdot \wti{\a}
= |\wti{\a}\cdot \wti{\a}|+h\cdot \wti{p}_*\wti \a\\
&=& |\wti{\a}\cdot \wti{\a}|+h\cdot \wti{\sigma}= |\wti{\a}\cdot \wti{\a}|+\|\wti \sigma\|_T.\ea \] The theorem in the case $b_1(N)> 2$ and $e$  non--torsion now follows from applying
 (\ref{equ:facts}).

To complete the proof for nontorsion $e$ let us notice that in the case $b_1(N)=1$ we have $b_2(M)=0$, hence there is nothing to prove.
The case $b_1(N) = 2$ (hence $b_{2}^{+}(M) = 1$) can be reduced to the case of $b_1(N) > 2$:
this follows from the fact that
manifolds with virtually RFRS fundamental group are either virtually abelian or have infinite virtual Betti number (see \cite[p.~271]{Ag08}). In the former case, as $N$ is closed, the only possibility by classical results is that $N$ is covered
by $T^3$.
In any case there exist a cover $\pi \colon  {\wti N} \to N$ with $b_1({\wti N}) > 2$; then application of the (in--) equalities of (\ref{equ:facts})  allow us to reduce the statement for $ x(\a)$ to the one for $ x(\pi^{*}\a)$.

The case that $e=0$ can be proved analogously (going to a cover with $b_{+}(M) > 1$ if necessary) using the equality $ SW_{S^1 \times N}(p^*\cc)
=  SW_{N}(\cc)$.
Finally the case that $e$ is torsion can easily be reduced to the case $e=0$: indeed, it is well--known that
any $S^1$--bundle $M$ over a 3--manifold $N$ with torsion Euler class is finitely covered by a product $S^1\times \wti{N}$. (See e.g. \cite[Proposition~3]{Bow09} or \cite[Theorem~2.2]{FV11a} for a quick proof.) By the argument above, the lower bounds on the complexity of surfaces in the finite cover $S^1\times \wti{N}$ translate into lower bounds on the complexity of the original manifold $M$.
\end{proof}

\begin{remarks}
1. If $e \in \Xi_N$, then the fact that Lemma \ref{above} does not apply  does not by any means entail that the refined adjunction inequality fails to hold; it is simply not possible for us, with the relatively limited knowledge of the Seiberg-Witten invariants of ${\wti N}$, to get more information than the one gathered above.\\
2. Conceivably one could get better bounds on $x$ by using any finite cover of $M$, and not just the ones induced by pull--back of a finite cover of $N$.
The argument of the proof of Theorem \ref{kronans} will show that such covers will not give any further information.
\end{remarks}

%=================================================================
\section{The minimal complexity function} \label{rest}

%================================================================
\subsection{Sharp bounds on the complexity of surfaces}\label{section:sharp}

In this section we will show that in many cases the inequality of Theorem \ref{mainthm} is in fact an equality.
This section partly builds on ideas from \cite{FV09}, which in turn was inspired by the results of \cite{Kr98}.

We start out with the following  definition from \cite{FV09}, which extends an earlier definition in \cite[Section~1.2]{Kr98} for product manifolds.
Given $\a \in H_2(M)$ we say that $\a$ has Property $(*)$ if there exists a
 (possibly disconnected) embedded  surface $\S \subset N$ and a (possibly disconnected) closed curve $c\subset N$ in general position with the following properties:
\bn
\item $\S$ is a Thurston norm minimizing surface dual to $p_*(\a)$,
\item the surface $\S$ lifts to a surface $\ti{\S}$ in $M$,
\item the singular surface $p^{-1}(c)\cup \ti{\S}$ represents $\a$,
\item the geometric intersection number of $\S$ and $c$ is given by the absolute value of the algebraic intersection number $\S\cdot c$.
\en
Note that, with appropriate orientations, we have $\a\cdot \a=2\,\S\cdot c$.
We have the following  lemma (cf. also \cite{Kr99}).

\begin{lemma}  \label{lem:starmin}
Let $\a\in H_2(M)$ that satisfies Property $(*)$, then there exists a surface $T$ which represents $\a$ with
\[ \chi_-(T)= |\a\cdot \a|+\|p_*\a\|_T.\]
\end{lemma}

\begin{proof}
Let $\S,\ti{\S}$ and $c$ as in the definition of Property ($*$).
 Around each singular point  of $p^{-1}(c)\cup \ti{\S}$ we can replace a pair of transverse disks with an embedded annulus having the same oriented boundary. Note that each replacement increases the complexity by 2. We therefore obtain a smooth surface $T$ representing the class dual to $\a$ with
\[ \chi_-(T)=\chi_-(\ti{\S})+\chi_-(p^{-1}(c))+2 |\S \cdot c|.\]
Note that $p^{-1}(c)$ is a union of tori, hence $\chi_-(p^{-1}(c))=0$. As
$2~\S\cdot c=\a\cdot \a$, we get
\[ \chi_-(T)= |\a\cdot \a|+\|p_*\a\|_T.\]
\end{proof}

We now have the following lemma.

\begin{lemma} \label{lem:getstar}
Let $M\to N$ be a principal $S^1$--bundle with Euler class $e$.
Let $\a \in H_2(M)$. We write $H:=H_1(N;\Z)/\mbox{torsion}$ and
$\sigma:=p_*\a\in H_2(N)\cong H^1(N;\Z)=\hom(H,\ll t\rr)$.
If  $\sigma(\Delta_N)\ne 0$, then $\a$ has Property $(*)$.
\end{lemma}

\begin{proof}
We write $n=\div(\sigma)$.
The assumption that $\sigma(\Delta_N)\ne 0$ implies by standard arguments (see e.g. \cite[Proposition~3.6]{FV08})
that the one variable Alexander polynomial of $N$ corresponding to $\sigma$ is non--zero.
It then follows from  \cite[Proposition~6.1]{McM02}
that there exists a connected Thurston norm minimizing surface $\S$ dual to the primitive class $\frac{1}{n}\sigma$.
By the Gysin sequence, $e\cdot \sigma=e\cdot p_*\a=0$.
The restriction of the $S^1$--bundle to $\S$ is therefore trivial, in particular
$\S$ lifts to a surface $\wti{\S}$ in $M$.
We now denote by $n\S$ and $n\wti{\S}$ the union of $n$ parallel copies of $\S$ and $\wti{\S}$.

It follows from the Gysin sequence  that we can find an embedded curve $c\subset N$ such that the class dual to
$\a$ is given by $[n\ti{\S}]+[p^{-1}(c)]$. We can assume that $c$ is in general position with $n\S$, hence $n\ti{\S}$ and $p^{-1}(c)$ are in general position.
It is not hard to see that as $\S$ is connected we can  choose (a possibly disconnected) $c$ such that the geometric intersection number of $\S$ and $c$ is given by the absolute value of $\S\cdot c$,
in particular such that the geometric intersection number of $n\S$ and $c$ is given by the absolute value of $n\S\cdot c$.
\end{proof}

\begin{remark}
Let $H$ be a torsion--free group and $p\in \Z[H]$ a non--zero polynomial.
We write $p=\sum_{h\in H}a_hh$. We consider the  Newton polyhedron of $p$,
i.e. the convex hull of $\{ h\in H\,|\, a_h\ne 0\}\subset H\otimes \R$. If for all faces $F$ of the Newton polyhedron we have
\[ \sum_{h\in F\cap H} a_h\ne 0,\]
then it is straightforward to see that for any $\phi:H\to \ll t\rr$ we have $\phi(p)\ne 0$.
\end{remark}

Note that the previous lemma together with Theorem \ref{mainthm} implies Corollary \ref{cor:mingen}, i.e. they imply that for `most' $S^1$--bundles over nice 3--manifolds with virtually RFRS fundamental group the equality
\[ x(\a)=|\a\cdot \a|+\|p_*\a\|_T\]
holds. The following question arises:

\begin{question}
Let $M$ be an $S^1$--bundle over a 3--manifold $N$.
Does the equality
\[ x(\a)=|\a\cdot \a|+\|p_*\a\|_T\]
hold for all $\a\in H_2(M)$?
\end{question}

Even though we just showed that we can answer the question in the affirmative for many cases,
we will now discuss an example that make us think that in general an affirmative answer is far from being granted.

\begin{lemma}
Let $N$ be a closed hyperbolic 3--manifold which admits a primitive class $\sigma\in H_2(N)$
which has a unique  Thurston norm minimizing surface $\S$ which consists of two homologically essential components $\S_1$ and $\S_2$.
Then we can find classes $\a\in H_2(S^1\times N)$, that can be chosen to have zero or positive self--intersection,
that do not have Property $(*)$.
\end{lemma}

\begin{proof}
Let $N$ be a closed hyperbolic 3--manifold which admits a primitive class $\sigma\in H_2(N)$
which has a unique  Thurston norm minimizing surface $\S$ which consists of two homologically essential components $\S_1$ and $\S_2$.
We denote by $\sigma_i, i=1,2$ the homology class represented by $\S_i$.
Note that $\S_1,\S_2$ can not be tori since we assumed that $N$ is hyperbolic.

\begin{claim}
The classes $\sigma_1$ and $\sigma_2$ define linearly independent elements in $H_2(N)$.
\end{claim}

If $\sigma_1$ and $\sigma_2$ do not define linearly independent elements, then
$\sigma_1,\sigma_2$ span an infinite cyclic subgroup which clearly contains $\sigma$.
Since $\sigma$ is primitive it follows that $\sigma_i=k_i\sigma$ for some $k_i\in \Z\sm \{0\}$.
But then
\[ \chi_-(\S_1)\geq \|\sigma_1\|_T=\|k_1\sigma\|_T=|k_1|(\chi_-(\S_1)+\chi_-(\S_2))>\chi_-(\S_1),\]
which is a contradiction.

Since $\sigma_1$ and $\sigma_2$ are linearly independent we can by Poincar\'e duality find oriented curves $c_i, i=1,2$ such
that $\sigma_i\cdot c_i=l_i$ for some $l_i\ne 0\in \N $ and such that $\sigma_i\cdot c_{3-i}=0$. Fix now $m \in \N$,
and consider $c$ the union of $l_2+m$ parallel copies of $c_1$ and $l_1$ copies of $-c_2$.
Note that
$c\cdot \S_1=(l_2+m)l_1$, $c\cdot \S_2=-l_1l_2$ and thus $c\cdot \S=ml_1$.

Note that any curve $d$ representing $[c]$ has to intersect $\S_1$ at least $(l_2+m)l_1$ times
and $\S_2$ at least $l_1l_2$ times. It now follows from the uniqueness property of $\S=\S_1\cup \S_2$ and the above discussion
that the class
in $H_2(S^1\times N)$ which corresponds to $([c],\sigma)$ under the K\"unneth isomorphism
$H_2(S^1\times N)\cong H_1(N)\oplus H_2(N)$ does not have Property $(*)$.
The lemma now follows from $\a\cdot \a=2\,c\cdot \S=2ml_1$, which depending on the choice of $m$ will be either positive or zero.
\end{proof}

In the proof of the previous lemma we identified in particular
classes $\a \in H_2(S^1\times N)$ such that Kronheimer's lower bounds (and our bounds if $\pi_1(N)$ is virtually RFRS)
imply that
\[ x(\a)\geq |\a\cdot \a|+\|p_*\a\|_T=2ml_1+\chi(\S).\] As the construction of Lemma \ref{lem:starmin} fails to apply,
it is not clear whether there exists a surface in $S^1\times N$ which represents $\a$ and which has complexity less than $4l_1l_2+2ml_1+\chi(\S)$.

If such a surface fails to exists, the following lemma would have the rather remarkable consequence that $x'$ is not as well-behaved as the Thurston norm:

\begin{lemma}
Let $N$ be a closed 3--manifold such that $\pi_1(N)$ is virtually RFRS and suppose that there exists a class $\a \in H_2(S^1\times N)$
such that
\[ x(\a)> |\a\cdot \a|+\|p_*\a\|_T,\]
then either the  function $x'$ does not define a seminorm on a finite cover of $S^1\times N$ or $x'$ is not multiplicative under finite covers.
\end{lemma}

\begin{proof}
Let $N$ be a closed 3--manifold such that $\pi_1(N)$ is virtually RFRS and suppose that there exists a class $\a\in H_2(S^1\times N)$
such that
\[ x(\a)> |\a\cdot \a|+\|p_*\a\|_T.\]
Let us assume that  the function $x'$ defines a seminorm on all  finite covers of $S^1\times N$ and that $x'$ is  multiplicative under finite covers.

Since $\pi_1(N)$  is virtually RFRS, there exists,  by Agol's theorem,   a finite cover $\pi:\wti{N}\to N$ such that $\pi^*(p_{*} \a)$ is quasi--fibered.
In particular there exists an open cone $V\subset H_2(S^1\times \wti{N};\R)$ with the following two properties:
\bn
\item for any class $\wti{\b}\in V\cap H_2(S^1\times \wti{N};\Z)$ we have
\[ x(\wti{\b})=|\wti{\b}\cdot \wti{\b}|+\|{\wti p}_*{\wti{\b}}\|_T \]
(see \cite[Corollary~2]{Kr99} for details);
\item $\pi^*\a$ lies on the boundary of $V$.
\en
By our assumption $x'$ is a seminorm on $H_2(S^1\times \wti{N})$; it thus follows that $x'(\pi^*\a)=\|{\wti p}_* \pi^*\a\|_T$.
Similarly, as we assume that  $x'$ is multiplicative under finite covers,
it now  follows that $x'(\a)=\|p_*\a\|_T$, i.e.
\[ x(\a)=|\a\cdot \a|+\|p_*\a\|_T,\]
which is a contradiction to our hypothesis.
\end{proof}

%===================================================================
\subsection{Multiplicativity under finite covers} \label{kron}

The following theorem determines sufficient conditions under which the complexity function is multiplicative under finite covers and implies, together with Corollary \ref{cor:mingen}, Theorem \ref{kronansintro}.
As we take arbitrary finite covers, we will see from the proof that iteration of the argument of Theorem \ref{above} imposes us to exclude from the statement infinitely many choices of $e \in H^2(N)$, precisely the classes contained in $\Theta_{N}$.

\begin{theorem} \label{kronans} Let $N$ be a closed irreducible $3$-manifold with virtually RFRS fundamental group.
Let  $p \colon  M \to N$ be a circle bundle of Euler class $e \in H^2(N) \setminus \Theta_{N}$ and let $\a\in H_2(M)$ such that $ x(\alpha) = |\a \cdot \a| +  \|p_{*} \a \|_T$.
Then for any finite cover $\rho_{M} :\widehat{M}\to M$ we have $x(\rho_{M}^*\alpha)=\deg(\rho_{M})\cdot x(\alpha)$.
\end{theorem}

\begin{proof}
We will only consider in the proof the case where $b_1(N)>2$ and that $e$ is non--torsion.
The other cases are very similar and can be treated as in the proof of Theorem \ref{above}, and we will  thus leave them to the reader.

 Let $\rho_M\colon \widehat{M} \to M$ be any finite cover. As discussed in Section \ref{section:prelim}, $\widehat{M}$ carries an $S^1$--action with orbit space ${\widehat N}$, which in turns covers $N$ with covering map $\rho_N\colon \widehat{N} \to N$.

According to  Lemma \ref{lem:agol} there exists a finite cover $\tau\colon \wti{N} \to \widehat{N}$ where the transfer of any class in $H_2({\widehat N})$ is quasifibered. It is clear that the composition of the covering maps gives a finite cover $\pi_{N} \colon \wti{N} \to N$  which again has the property that  the transfer of any class in $H_2({N})$ is quasifibered.

We now denote by $\wti{M}$ the total space of the pull back under $\tau$ of the $S^1$--bundle $\widehat{M}\to \widehat{N}$.
In line with our previous notation, we denote as well with $\tau$ the $4$-dimensional covering map $\tau\colon \wti{M} \to \widehat{M}$. Also, we denote by $\pi_M\colon\wti{M}\to M$ the composition of the covers.
Finally we denote by $\wti{e}\in H^2(\wti{N})$  the Euler class of the $S^1$--bundle $\wti{M}\to \wti{N}$.

\begin{claim}
We have
\[ \frac{1}{\deg \pi_N } \pi_{N*} {\wti e} \notin {\mathcal E}_{N}.\]
\end{claim}

 As discussed in Section \ref{section:prelim}, the Euler class $\wti e$ of ${\wti p}\colon \wti{M} \to \wti{N}$ satisfies the equation $q {\wti e} = \pi_{N}^{*} e \in H^{2}({\wti N})$.
  The claim now follows from our assumption that the image of $e$ in $H^2(N;\R)$ is not the (integer) multiple of an element of ${\mathcal E}_{N}$, i.e. from our assumption that $e \notin \Theta_{N}$.

Now  let $\a\in H_2(M)$. We denote by $\widehat{\a} \in H^2({\widehat M})$ and $\wti{\a} \in H^{2}({\wti M})$ the corresponding classes given by the transfer maps.
From the proof of Theorem \ref{above} we see that the fact that
\[  \frac{1}{\deg \pi_N } \pi_{N*} {\wti e} \notin {\mathcal E}_{N}, \]
implies that there exists a class $h \in {\wti H}_{p_{*} \a}$
 such that the  inequality
  \[ x(\wti{\a})\geq |\wti{\a}\cdot \wti{\a}|+{\wti p}^{*} h \cdot \wti{\a}\]
   holds true. (Here and throughout the proof we adopt the notation of the proof of  Theorem \ref{above}.)
Now ${\wti p}^{*} h \cdot \wti{\a} = h \cdot {\wti p}_{*} \wti{\a}$ and using the equality ${\wti p}_{*} \wti{\a} = q \cdot \pi_{N}^{*} p_{*} \a \in H_{2}({\wti N};\R)$,
 with $q\in \N$ defined as in Section \ref{section:prelim}, this takes the form
 \[ \ba{rcl} x(\wti{\a})&\geq& |\wti{\a}\cdot \wti{\a}|+{\wti p}^{*} h \cdot \wti{\a}\\
 &\geq& |\wti{\a}\cdot \wti{\a}|+h \cdot {\wti p}_{*} \wti{\a}\\
 &=& |\wti{\a}\cdot \wti{\a}|+q \cdot (h\cdot \pi_{N}^{*} p_{*} \a)\\
 &=&   |\wti{\a}\cdot \wti{\a}|+ q\cdot  \| \pi_{N}^{*} p_{*} \a \|_{T},\ea \]
 where the last equality follows from  $h \in {\wti H}_{p_{*} \a}$.
 Using the equations in (\ref{equ:facts}) and the equality ${\deg \pi_M } =  q \, { \deg \pi_N }$, we can rewrite this as \begin{equation} \label{equ:adjcover} x(\wti{\a})\geq  {\deg \pi_M } ( |\a \cdot \a| + \| p_{*} \a \|_{T}). \end{equation}
Now  let $\a\in H_2(M)$ be any class such that \be \label{equ:xa} x(\a)=|\a\cdot \a|+\|p_*\a\|_T.\ee
By elementary reasons (see Equations (\ref{equ:facts})) we have
\[ x ({\wti \a}) \leq  {\deg \tau } \cdot x (\widehat{\a}) \leq {\deg \pi_M } \cdot x(\a).  \]
Playing this equation against Equations (\ref{equ:adjcover}) and (\ref{equ:xa}), we see that each inequality is in fact an equality, hence (as ${\deg \pi_M } = {\deg \tau } \cdot {\deg \rho_M }$)
we obtain multiplicativity of the complexity function as stated.

\end{proof}

%=========================================================

\end{document}